\documentclass[11pt]{amsart}
\usepackage[T1]{fontenc}
\usepackage[utf8]{inputenc}
\usepackage{amsmath, amsthm, amsfonts, amssymb}
\usepackage{fullpage}
\usepackage{booktabs}
\usepackage{colonequals}
\usepackage{caption}
\usepackage{graphicx}
\usepackage{mathtools}
\usepackage{enumitem}
\usepackage{algorithm}
\usepackage[noend]{algpseudocode}
\algrenewcommand\algorithmicdo{}   
\usepackage{tikz}
\usetikzlibrary{arrows.meta} 
\usepackage{array}
\usepackage{float}
\usepackage[colorlinks=true,allcolors=blue]{hyperref}
\usepackage[f]{esvect}
\allowdisplaybreaks

\newcommand{\curvyarrows}{%
  \begin{tikzpicture}[baseline=-0.4ex, shorten <=1pt, shorten >=1pt]
    \draw[-{Stealth[length=2mm]}] (0,0.12) to[bend left=12] (0.7,0.12);
    \draw[-{Stealth[length=2mm]}] (0.7,-0.12) to[bend left=12] (0,-0.12);
  \end{tikzpicture}%
}

\numberwithin{equation}{section}
\newtheorem{thm}{Theorem}[section]
\newtheorem{lem}[thm]{Lemma}
\newtheorem{prop}[thm]{Proposition}
\newtheorem{cor}[thm]{Corollary}

\theoremstyle{definition}
\newtheorem{defn}[thm]{Definition}

\theoremstyle{remark}
\newtheorem{rem}[thm]{Remark}

\newtheorem{claim}[thm]{Claim}

\DeclareMathOperator{\Var}{Var}
\DeclareMathOperator{\Cov}{Cov}
\newcommand{\E}{\mathbb{E}}
\newcommand{\Prb}{\mathbb{P}}

\title{Improved lower bounds for the maximum order of an induced acyclic subgraph}

\author{Shamil Asgarli}
\address{Dept. of Mathematics \& Computer Science, Santa Clara University, Santa Clara, CA 95053.}\email{sasgarli@scu.edu}

\author{Donald Falkenhagen}
\address{Monte Sereno, CA 95030} \email{donaldfalkenhagen@gmail.com}

\author{Kaya Hoshi}
\address{School of Engineering, Santa Clara University, Santa Clara, CA 95053.}
\email{khoshi@scu.edu}


\begin{document}

\begin{abstract}
Computing the cardinality of a maximum induced acyclic vertex set in a digraph is NP-hard. Since finding an exact solution is computationally difficult, a fruitful approach is to establish high-quality lower bounds that are efficiently computable. We build on the Akbari--Ghodrati--Jabalameli--Saghafian (AGJS) bound for digraphs by adapting refinement techniques used by (a) Selkow and Harant--Mohr and (b) Angel--Campigotto--Laforest in their respective improvements of the Caro--Wei bound for undirected graphs. First, inspired by (a), we prove a neighborhood-based refinement of the AGJS bound that incorporates local degree data of each vertex. Second, inspired by (b), we compute the variance of the size of a feedback vertex set returned by a randomized algorithm. This result, combined with the Bhatia--Davis inequality, yields a tighter lower bound than the AGJS bound. 
\end{abstract}

\maketitle

\section{Introduction}

Finding the independence number $\alpha(G)$ of an undirected graph (the size of a maximum set of pairwise nonadjacent vertices) is a foundational problem in extremal graph theory. Computing $\alpha(G)$ is one of the classic NP-hard problems. A classical result of Tur\'an~\cite{turan1941} gives the lower bound $\alpha(G) \ge n/(1+\bar{d})$, where $\bar{d}$ is the average degree. As a refinement of Tur\'an's bound, Caro \cite{caro1979} and Wei \cite{wei1981} independently discovered an elegant lower bound in terms of the degree sequence:
\begin{equation}\label{eq:CaroWei}
\alpha(G) \;\geq\; \sum_{v\in V(G)} \frac{1}{1+d(v)}.
\end{equation}
The standard proof is a direct application of the probabilistic method \cite{alon2016}: consider a uniformly random permutation $\pi$ of the vertices $V(G)$. Construct an independent set $I$ by including every vertex $v$ that appears before all of its neighbors in $\pi$. By linearity of expectation, the expected size of this set is precisely the sum on the right-hand side of \eqref{eq:CaroWei}. Since the maximum size is at least the expected size of a randomly generated instance, the bound \eqref{eq:CaroWei} follows.

The Caro--Wei bound inspired more sophisticated bounds:
\begin{itemize} \item \textbf{Neighborhood-based refinements:} Selkow \cite{selkow1994} proposed strengthening the Caro--Wei bound by applying it again to the residual graph formed by vertices at distance at least two from the initial independent set $I$. Harant and Mohr \cite{harant2017} identified a subtle error in Selkow's argument and provided a correct proof of the following bound:
\[
\alpha(G) \ge \sum_{v\in V(G)} \frac{1}{1+d(v)}\left(1+\max\left\{0, \frac{d(v)}{d(v)+1} - \sum_{u\in N(v)}\frac{1}{1+d(u)}\right\}\right).
\]
The improvement is based on the local degree distribution around each vertex.
\item \textbf{Variance-based refinements:} Another approach examines the second moment of the size of a randomly generated set. Angel, Campigotto, and Laforest \cite{acl2013} analyzed a random process (the LL algorithm, which stands for ``ListLeft'' \cite{acl2011}) for producing vertex covers. This is complementary to constructing independent sets. By applying the Bhatia--Davis inequality \cite{bhatia2000} to the size of the random vertex cover $S$ generated by this algorithm, they derived the bound:
\[
\alpha(G) \geq \left(\sum_{v\in V(G)}\frac{1}{1+d(v)}\right) + \frac{\Var(|S|)}{\left(\sum_{v\in V(G)}\frac{1}{1+d(v)}\right) - c},
\]
where $\Var(|S|)$ is the variance of the random cover size and $c$ is the number of connected components of $G$.
\end{itemize}
Both refinements go beyond the first-moment method and use additional graph structure to tighten the bound \eqref{eq:CaroWei}.

\subsection*{Extension to digraphs}
These concepts extend naturally to directed graphs, where the analog of an independent set is an \emph{acyclic set} (a set of vertices that induces a subgraph with no directed cycles). The cardinality of the largest acyclic set is denoted by $\vv{\alpha}(D)$. Computing $\vv{\alpha}(D)$ is also NP-hard. 
Akbari, Ghodrati, Jabalameli, and Saghafian (AGJS) \cite{akbari2017} established the following directed analog of the Caro--Wei bound:
\begin{equation}\label{eq:AGJS}
\vv{\alpha}(D) \;\ge\; \sum_{v\in V(D)} \rho_D(v).
\end{equation}
The term $\rho_D(v)$ arises from a probabilistic argument on a random vertex permutation. A vertex $v$ is selected to be in the acyclic set if it appears before all of its out-neighbors (event $A_v$) \emph{or} before all of its in-neighbors (event $B_v$). The probability of this union event is computed by:
\begin{equation}\label{eq:rho_D(v)}
\rho_D(v) \colonequals \Prb(A_v \cup B_v) = \Prb(A_v) + \Prb(B_v) - \Prb(A_v \cap B_v) = \frac{1}{1+d_v^+} + \frac{1}{1+d_v^-} - \frac{1}{1+d_v},
\end{equation}
where $d_v^+$, $d_v^-$, and $d_v$ are the out-degree, in-degree, and total degree of the vertex $v$, respectively (see Preliminaries for precise definitions). Crucially, the set of vertices satisfying the event $A_v \cup B_v$ forms a random acyclic set. We include a self-contained proof in Section~\ref{sec:selkow} and use this construction as our starting point. The bound \eqref{eq:AGJS} is the first-order estimate for $\vv{\alpha}(D)$, playing the same role as the Caro--Wei bound does for $\alpha(G)$. For comparison, Gruber \cite{Gru11} proved the weaker bound:
\[
\vv{\alpha}(D) \;\ge\; \sum_{v\in V(D)} \frac{1}{1+d_v^{+}}.
\]

\subsection*{Our contributions}
To the best of our knowledge, the aforementioned refinements for the Caro--Wei bound have not been extended to the directed graph setting. This paper fills that gap. We obtain both the neighborhood-based and variance-based improvements, yielding two new lower bounds on $\vv{\alpha}(D)$ that strengthen the AGJS bound \eqref{eq:AGJS}. 
\begin{itemize}[leftmargin=1.7em]
\item \textbf{Neighborhood-based approach.} Following Selkow's idea \cite{selkow1994} (with the correction of Harant--Mohr \cite{harant2017}), we prove a neighborhood-based improvement for digraphs by applying the bound \eqref{eq:AGJS} to a random residual subgraph. See Theorem~\ref{thm:main_selkow}.
\item \textbf{Variance-based approach.} Inspired by Angel--Campigotto--Laforest \cite{acl2013}, we define the \emph{DLL algorithm}, a randomized procedure for generating a feedback vertex set $S$. We show that the probability a vertex is \emph{not} in $S$ equals $\rho_D(v)$, thus linking our algorithm to the bound \eqref{eq:AGJS}. We then derive an exact formula for $\Var(|S|)$ via a detailed covariance analysis and apply the Bhatia--Davis inequality \cite{bhatia2000} to obtain our second refinement. See Theorem~\ref{thm:main_acl}.
\end{itemize}

\subsection*{Preliminaries} Before proceeding, we formally define the terminology used throughout this paper. We work with finite simple loopless digraphs: for any ordered pair of vertices $(u,v)$, at most one arc $u\to v$ is allowed.\footnote{We use the term \emph{arc} for a directed edge; an arc $(u,v)$ means the edge $u\to v$.} In particular, both arcs $u\to v$ and $v\to u$ may be present simultaneously. Given a digraph $D$, we write $V(D)$ and $A(D)$ for the vertex and arc sets, respectively. 

For a vertex $v$, its in-neighborhood, out-neighborhood, and neighborhood are denoted $N^-(v)\colonequals\{w \mid (w,v)\in A(D) \}$, $N^+(v)\colonequals\{w \mid (v, w)\in A(D)\}$, and $N(v)\colonequals N^-(v)\cup N^+(v)$, respectively. The \emph{in-degree} $d_v^-$ and \emph{out-degree} $d_v^+$ are the cardinalities of these neighborhoods. We define the \emph{total degree} of $v$, denoted $d_v$, as the size of its neighborhood: $d_v \colonequals |N(v)|$. For conciseness, we use the subscript notation $d_v$ rather than $d(v)$. Note that this definition may differ from the sum $d_v^-+d_v^+$, since we allow directed 2-cycles (that is, a vertex $u$ may belong to both $N^-(v)$ and $N^+(v)$). Defining $d_v$ as the neighborhood size (rather than the number of incident arcs) ensures that the probabilistic formulas in this paper are valid for all simple digraphs. For a set of vertices $S\subseteq V(D)$, its neighborhood is:
$$
N(S) \colonequals \bigcup_{x\in S} N(x).
$$
The \emph{weakly connected components} of $D$ are the connected components of its underlying undirected graph; for brevity, we will simply call them \emph{connected components}.

A vertex subset $T \subseteq V(D)$ is called an \emph{acyclic set} if the induced subgraph $D[T]$ contains no directed cycle. The maximum size of such a set is denoted by $\vv{\alpha}(D)$. A \emph{feedback vertex set} (FVS) is a set of vertices $S$ such that $D-S$ is acyclic; the minimum size of such a set is denoted $\vv{\beta}(D)$. These two complementary parameters satisfy the identity
\begin{equation}\label{eq:complementary-identity}
\vv{\alpha}(D)+\vv{\beta}(D)=|V(D)|.
\end{equation} 
\noindent\textbf{Terminology.} In some sources, ``Maximum Acyclic Subgraph'' refers to an \emph{arc}-deletion problem; throughout, we study the \emph{vertex}-induced version. For a comprehensive treatment of digraph theory, we refer the reader to Bang-Jensen and Gutin~\cite{bangjensen2009}.

The \emph{dichromatic number}, introduced by Neumann-Lara \cite{neumannlara1982}, is the minimum number of acyclic color classes needed to partition $V(D)$. While we do not study colorings here, the quantity $\vv{\alpha}(D)$ is an extremal parameter in that theory (that is, the maximum possible size of a color class), analogous to the independence number in traditional graph coloring. For recent related work on acyclic and dichromatic number bounds for oriented graphs, see Aboulker, Havet, Pirot, and Schabanel~\cite{aboulker2025}.

\subsection*{Outline of the paper}
The remainder of the paper is organized as follows. Section~\ref{sec:selkow} develops the neighborhood-based lower bound for $\vv{\alpha}(D)$. Section~\ref{sec:DLL} introduces our DLL algorithm and proves its fundamental properties. Section~\ref{sec:variance} provides the variance decomposition and a complete covariance catalog for the DLL algorithm. Section~\ref{sec:bd} assembles the variance-based lower bound for $\vv{\alpha}(D)$. Finally, Section~\ref{sec:experiments} offers numerical experiments on certain types of random graphs to illustrate the behavior of our improvements.

\subsection*{Acknowledgment} We are grateful to the referee for useful references and helpful comments that improved the exposition. We also thank the referee for suggesting that we expand our experiments, which led us to add Subsection~\ref{subsec:scale-free}.

\section{A neighborhood-based refinement for digraphs}\label{sec:selkow}

The main idea of the neighborhood-based improvement is to apply the AGJS bound~\eqref{eq:AGJS} to a \emph{residual subgraph}, namely, the subgraph that remains after removing an initial acyclic set and its neighbors. Since we consider degrees in a subgraph, we first need to understand how the quantity $\rho_D(v)$ from \eqref{eq:rho_D(v)} behaves as degrees decrease.

\begin{lem}[Monotonicity]\label{lem:monotonicity} Consider the function 
\[
f(x,y, z)\colonequals \frac{1}{1+x}+\frac{1}{1+y}-\frac{1}{1+x+y-z} 
\]
on the domain $x\geq 0$, $y\geq 0$, $0\leq z\leq \min(x,y)$. Suppose $a, b, c, d, s, t$ are nonnegative integers such that $a\leq c$, $b\leq d$, and $s\leq t$. Then $f(a,b,s) \geq f(c,d,t)$.
\end{lem}

\begin{proof}
We compute the partial derivatives:
\begin{align*}
\frac{\partial f}{\partial x} &\;=\; -\frac{1}{(1+x)^2} + \frac{1}{(1+x+y-z)^2} \;\le\; 0, \\ 
\frac{\partial f}{\partial y} &\;=\; -\frac{1}{(1+y)^2} + \frac{1}{(1+x+y-z)^2} \;\le\; 0. \\
\frac{\partial f}{\partial z} &\;=\; -\frac{1}{(1+x+y-z)^2} \;\le\; 0.
\end{align*}
Thus, $f$ is nonincreasing in each variable. Since $a\le c$, $b\le d$, and $s\le t$, we conclude that $f(a,b,s)\geq f(c,d,t)$.
\end{proof}

The function in Lemma~\ref{lem:monotonicity} matches $\rho_D(v)$ in \eqref{eq:rho_D(v)}. For a vertex $v$, we have $d_v=|N(v)|=d_v^{+} + d_v^{-}-|N^+(v)\cap N^-(v)|$ since by convention we count a neighbor once, even if it is both an in- and an out-neighbor. Setting $x=d_v^+$, $y=d_v^-$, and $z=|N^+(v)\cap N^-(v)|$ gives $\rho_D(v)=f(x,y,z)$.

We now prove one of the main results of our paper.

\begin{thm}[Neighborhood-based refinement]\label{thm:main_selkow}
Let $D$ be a digraph and let $\rho_D$ be the function defined in \eqref{eq:rho_D(v)}. Then
\begin{equation}\label{eq:selkow-digraph}
\vv{\alpha}(D)\;\ge\;\sum_{v\in V(D)} \rho_D(v)\left(1+\max\left\{0,\;1-\rho_D(v)-\sum_{u\in N(v)}\rho_D(u)\right\}\right).
\end{equation}
\end{thm}

\begin{proof}
Let $\pi$ be a uniformly random permutation of $V(D)$. For each $v\in V(D)$, define the events:
\[
A_v \colonequals \text{all } N^+(v)\text{ appear after } v \text{ in } \pi, \qquad
B_v \colonequals \text{all } N^-(v)\text{ appear after } v \text{ in } \pi.
\]
Let $I$ be the random set of vertices for which the event $A_v\cup B_v$ occurs.

We claim that $I$ is acyclic. We show that no directed cycle can have all its vertices in $I$. Let $C = (v_1 \to v_2 \to \dots \to v_k \to v_1)$ be an arbitrary directed cycle in the graph $D$. Consider the vertex on this cycle, say $v_j$, that appears \emph{last} in the random permutation $\pi$. We show that $v_j\notin I$.

For $v_j$ to be included in $I$, the event $A_{v_j} \cup B_{v_j}$ must occur. However, both events are impossible:
\begin{itemize}
    \item The event $A_{v_j}$ requires all out-neighbors of $v_j$ to appear \emph{after} $v_j$ in $\pi$. But the successor of $v_j$ on the cycle, $v_{j+1}$ (indices modulo $k$), is an out-neighbor that must appear \emph{before} $v_j$ in the permutation, since $v_j$ is the last vertex of the cycle to appear in $\pi$. Thus, $A_{v_j}$ cannot occur.
    
    \item The event $B_{v_j}$ requires all in-neighbors of $v_j$ to appear \emph{after} $v_j$ in $\pi$. But the predecessor of $v_j$ on the cycle, $v_{j-1}$ (indices modulo $k$), is an in-neighbor that must also appear \emph{before} $v_j$ in the permutation, since $v_j$ is the last vertex of the cycle to appear in $\pi$. Thus, $B_{v_j}$ cannot occur.
\end{itemize}
Thus, the union event $A_{v_j}\cup B_{v_j}$ cannot occur, meaning $v_j\notin I$. We have shown that every directed cycle $C$ in the graph has at least one vertex that is not in $I$. Therefore, $I$ does not induce any directed cycles, that is, $I$ is acyclic.

Next, we apply linearity of expectation and formula \eqref{eq:rho_D(v)} for $\Prb(A_v\cup B_v)$ to obtain:
\begin{equation}\label{eq:E-A1}
\E[|I|] \;=\; \sum_{v\in V(D)} \Prb(A_v\cup B_v) \;=\; \sum_{v\in V(D)} \rho_D(v).
\end{equation}
Now, define the random residual subgraph $H$ by
\[
V(H)\;=\; V(D)\setminus\bigl(I\cup N(I)\bigr).
\]
For any $\pi$, the sets $I$ and $V(H)$ are disjoint and there are no arcs between them. Thus, the union of an acyclic set from $D[I]$ and an acyclic set from $H$ is acyclic in $D$. This yields the inequality:
\[
\vv{\alpha}(D)\;\ge\; |I| + \vv{\alpha}(H).
\]
Taking expectations, we have
\begin{equation}\label{eq:alpha-lb}
\vv{\alpha}(D) \;\ge\; \E[|I|] + \E[\vv{\alpha}(H)] = \sum_{v\in V(D)} \rho_D(v) + \E[\vv{\alpha}(H)].
\end{equation}
We now bound $\E[\vv{\alpha}(H)]$ from below by applying the bound \eqref{eq:AGJS} to the random subgraph $H$. Observe that $\rho_D(v)$ can be expressed as:
\[
\rho_D(v) = \frac{1}{1+d_v^{+}} + \frac{1}{1+d_v^{-}} -  \frac{1}{1+d_v^{+}+d_v^{-}-t_{v}} 
\]
where $t_v = |N^{+}(v)\cap N^{-}(v)|$. For any vertex $v\in V(H)$, its in- and out-degree in $H$ do not exceed its in- and out-degree in $D$; moreover, $|N^{+}_{H}(v)\cap N^{-}_{H}(v)| \leq |N^{+}_{D}(v)\cap N^{-}_{D}(v)|$. Therefore, by Lemma~\ref{lem:monotonicity}, we have $\rho_H(v) \ge \rho_D(v)$. Thus,
\begin{align}
\E[\vv{\alpha}(H)]
&\ge \E\Bigl[\sum_{v\in V(H)} \rho_H(v)\Bigr]
\;\ge\; \E\Bigl[\sum_{v\in V(H)} \rho_D(v)\Bigr]
\;=\; \sum_{v\in V(D)} \rho_D(v)\cdot \Prb(v\in V(H)). \label{eq:EH-bound}
\end{align}
Combining \eqref{eq:alpha-lb} and \eqref{eq:EH-bound} gives
\begin{equation}\label{ineq:selkow-alpha(D)}
\vv{\alpha}(D)\;\ge\; \sum_{v\in V(D)} \rho_D(v)\,\Bigl(1 + \Prb(v\in V(H))\Bigr).
\end{equation}
The final step is to bound $\Prb(v\in V(H))$ from below. A vertex $v$ is \emph{not} in $V(H)$ if and only if $v \in I \cup N(I)$. Using the union bound,
\[
\Prb(v\notin V(H)) \;\le\; \Prb(v\in I) + \sum_{u\in N(v)} \Prb(u\in I)
\;=\; \rho_D(v) + \sum_{u\in N(v)} \rho_D(u).
\]
This implies
\begin{equation}\label{eq:Prb(v-in-H)}
\Prb(v\in V(H)) \;=\; 1 - \Prb(v\notin V(H)) \;\ge\; 1 - \rho_D(v) - \sum_{u\in N(v)} \rho_D(u).
\end{equation}
Combining the trivial inequality $\Prb(v\in V(H))\geq 0$ with \eqref{eq:Prb(v-in-H)}, we obtain:
\begin{equation}\label{eq:v-in-H}
\Prb(v\in V(H))\;\ge\; \max\left\{0, \; 1 - \rho_D(v) - \sum_{u\in N(v)} \rho_D(u)\right\}.
\end{equation}
Finally, substituting the bound \eqref{eq:v-in-H} for $\Prb(v\in V(H))$ back into our main inequality \eqref{ineq:selkow-alpha(D)} for $\vv{\alpha}(D)$ yields
\[
\vv{\alpha}(D)\;\ge\;\sum_{v\in V(D)} \rho_D(v)\Bigl(1+\max\{0,\;1-\rho_D(v)-\sum_{u\in N(v)}\rho_D(u)\}\Bigr),
\]
which is precisely the inequality \eqref{eq:selkow-digraph} claimed in the theorem.
\end{proof}

\section{A variance-based approach via feedback vertex sets}\label{sec:DLL}

Our next goal is to develop a variance-based refinement in the spirit of Angel, Campigotto, and Laforest (ACL)~\cite{acl2013}. This requires a random variable whose expectation recovers the AGJS bound \eqref{eq:AGJS} from Akbari--Ghodrati--Jabalameli--Saghafian~\cite{akbari2017}. Our strategy is to construct such a variable by analyzing the problem from a complementary perspective: instead of building a random \emph{acyclic set}, we define a randomized procedure that generates a feedback vertex set (FVS). 

To implement our strategy, this section introduces the DLL algorithm.\footnote{The name ``DLL algorithm'' is chosen as a counterpart to the ``LL algorithm'' from the undirected setting \cite{acl2013}; ``DLL'' is an abbreviation for \textbf{D}irected \textbf{L}ist\textbf{L}eft.} The analysis is organized as follows. In \S\ref{subsec:dl_definition}, we formally define the algorithm and prove its correctness (that it always produces an FVS). In \S\ref{subsec:dl_structural}, we establish its key structural properties, showing that it produces optimal solutions and deriving a worst-case bound on its output size. Finally, in \S\ref{subsec:dl_akbari}, we compute the expected FVS size and show that it recovers the AGJS bound \eqref{eq:AGJS}.

\subsection{The DLL algorithm and its correctness}\label{subsec:dl_definition}

Our algorithm applies a simple local rule to the vertices based on a random permutation; we formalize this as a labeling. Throughout this section, we define a \emph{labeling} as a bijection $L\colon V(D)\to \{1,\dots,n\}$ where $n=|V(D)|$ is the number of vertices. We imagine a labeling as arranging the vertices in a line, with labels increasing from left to right.

\begin{defn}[Right neighbors for a labeling]
Given a labeling $L\colon V(D)\to \{1,\dots,n\}$, a neighbor $w$ of a vertex $u$ is called a \emph{right in-neighbor} if $w \to u$ is an arc and $L(u) < L(w)$. Similarly, $w$ is a \emph{right out-neighbor} if $u \to w$ is an arc and $L(u) < L(w)$.

We denote the set of all right in-neighbors and right out-neighbors of $u$ by $R^-(u)$ and $R^+(u)$, respectively. Formally, these sets are given by:
\[
R^-(u) \colonequals \{w \in N^-(u) \mid L(u) < L(w) \}
\]
and
\[
R^+(u) \colonequals \{w \in N^+(u) \mid L(u) < L(w) \}.
\]
\end{defn}

The DLL algorithm uses the existence of both types of right neighbors to decide which vertices to place in its output set.

\begin{defn}[DLL algorithm]
Let $L$ be a uniformly random labeling of $V(D)$. The \emph{DLL algorithm} outputs a set $S\subseteq V(D)$ defined by the following rule for each vertex $u$:
\[
u\in S \quad \Longleftrightarrow \quad
\bigl(\exists\ \text{a right in-neighbor of $u$}\bigr)\ \text{and}\ \bigl(\exists\ \text{a right out-neighbor of $u$}\bigr).
\]
\end{defn}

\begin{algorithm}[H]
\caption{DLL: A randomized FVS procedure (input: digraph $D$)}
\begin{algorithmic}[1]
\State Let $L\colon V(D)\to\{1,\dots,n\}$ be a uniform random labeling.
\State $S\gets \varnothing$.
\ForAll{$u\in V(D)$}
  \If{$(\exists w\in N^-(u): L(u)<L(w))$ \textbf{and} $(\exists w\in N^+(u): L(u)<L(w))$}
    \State $S\gets S\cup\{u\}$.
  \EndIf
\EndFor
\State \Return $S$.
\end{algorithmic}
\end{algorithm}

First, we confirm that the procedure is correct: its output is always a feedback vertex set.

\begin{lem}[DLL produces an FVS]\label{lem:DLisFVS}
For any labeling $L$, the set $S$ returned by the DLL algorithm is a feedback vertex set.
\end{lem}

\begin{proof}
We show directly that the subgraph $D-S$ is acyclic. It suffices to show that for any directed cycle $\gamma$ in the original graph $D$, at least one of its vertices must belong to $S$.

Let $\gamma$ be an arbitrary directed cycle in $D$. Consider the vertex on this cycle, say $v$, that has the \emph{minimal} label under the given labeling $L$. Since $v$ is part of a cycle, it must have an in-neighbor $x$ and an out-neighbor $y$ that are also on $\gamma$. By our choice of $v$ as the vertex with the minimum label on the cycle, it must be that $L(v) < L(x)$ and $L(v) < L(y)$.
\begin{itemize}
    \item Since $x \to v$ is an arc and $L(v) < L(x)$, it follows that $x$ is a right in-neighbor of $v$.
    \item Since $v \to y$ is an arc and $L(v) < L(y)$, it follows that $y$ is a right out-neighbor of $v$.
\end{itemize}
By the definition of the DLL algorithm, any vertex that has both a right in-neighbor and a right out-neighbor is included in the set $S$. Therefore, $v \in S$. 

Since we have shown that every directed cycle in $D$ contains at least one vertex that must be in $S$, no cycle can exist within the complement $D-S$. Thus, $D-S$ is acyclic, and $S$ is a feedback vertex set.
\end{proof}

\subsection{Structural properties of the DLL algorithm}\label{subsec:dl_structural} The DLL algorithm can produce a minimum FVS. The proof of this property relies on the following elementary lemma regarding the structure of \emph{minimal} feedback vertex sets, which we state here for completeness.

\begin{lem}[Criticality of minimal FVS]\label{lem:critical}
If $S^\ast$ is a minimal feedback vertex set (with respect to inclusion), then for every $v\in S^\ast$, the digraph $D-(S^\ast\setminus\{v\})$ contains a directed cycle that passes through $v$.
\end{lem}

\begin{proof}
Let $v \in S^\ast$. By the minimality of $S^\ast$, the set $S' \colonequals S^\ast\setminus\{v\}$ is not a feedback vertex set. This means the subgraph $D-S'$ must contain at least one directed cycle. Any such cycle must necessarily pass through the vertex $v$, because all other vertices of $S^\ast$ are still in $S'$, and any cycle not involving $v$ would have already existed in the acyclic graph $D-S^\ast$.
\end{proof}

With these preliminaries, we can now establish the fundamental properties of the DLL algorithm. We begin by showing that the algorithm is powerful enough, in principle, to find an optimal solution.

\begin{prop}[Optimality of the DLL algorithm]\label{prop:DL_optimality}
For any digraph $D$, there exists a labeling $L^\ast$ for which the DLL algorithm returns a minimum feedback vertex set.
\end{prop}

\begin{proof} The proof is constructive: given a minimum feedback vertex set $S^{\ast}$, we construct a specific labeling $L^\ast$ and show that the DLL algorithm with this labeling returns exactly the set $S^\ast$. To define the labeling $L^\ast$, we partition the vertices $V(D)$ into two sets, $S^\ast$ and its complement $V(D)\setminus S^\ast$, and assign labels as follows:
\begin{itemize}
    \item Vertices $v \in S^\ast$ are assigned the smallest available labels, that is,
    \[ L^\ast(v) \in \{1, 2, \dots, |S^\ast|\}. \]
    \item The remaining vertices $u \in V(D)\setminus S^\ast$ induce an acyclic subgraph. We assign them the labels from $\{|S^\ast|+1, \dots, n\}$ according to a topological sort, which guarantees that $L^\ast(x) < L^\ast(y)$ for any arc $x \to y$ induced by $V(D)\setminus S^\ast$.
\end{itemize}
We analyze the output $S$ of the DLL algorithm with the labeling $L^\ast$ by considering the two cases.

\textbf{Case 1:} A vertex $u \notin S^\ast$.

We must show that DLL does not include $u$ in its output. It suffices to show that $u$ has no right in-neighbor with respect to the labeling $L^{\ast}$. Consider any in-neighbor $w \in N^-(u)$. There are two possibilities for $w$:
    \begin{itemize}
        \item If $w \in S^\ast$, then by our construction $L^\ast(w) \le |S^\ast| < L^\ast(u)$. Thus, $w$ is not a right in-neighbor of $u$.
        \item If $w \notin S^\ast$, then the arc $w \to u$ is induced by the acyclic set $V(D)\setminus S^\ast$. By the property of the topological sort, we must have $L^\ast(w) < L^\ast(u)$. Thus, $w$ is not a right in-neighbor of $u$.
    \end{itemize}
    In either case, $u$ has no right in-neighbors. Therefore, the algorithm does not add $u$ to $S$.

\textbf{Case 2:} A vertex $v \in S^\ast$.

We must show that DLL adds $v$ to its output. By Lemma~\ref{lem:critical}, since $S^\ast$ is a minimal FVS, there exists a directed cycle in $D-(S^\ast\setminus\{v\})$ that passes through $v$. Let $u \to v \to w$ be arcs on such a cycle. The vertices $u$ and $w$ are not in $S^\ast\setminus\{v\}$. Since $u$ and $w$ are neighbors of $v$, we infer that $u\neq v$ and $w\neq v$, which means $u, w \notin S^\ast$. Accordingly, our labeling construction guarantees the following ordering:
    \[ L^\ast(v) \le |S^\ast| \quad \text{and} \quad L^\ast(u), L^\ast(w) > |S^\ast|. \]
    This directly implies that $L^\ast(v)<L^\ast(u)$ and $L^\ast(v)<L^\ast(w)$. Therefore, $u$ is a right in-neighbor of $v$; similarly, $w$ is a right out-neighbor of $v$. As a result, the algorithm adds $v$ to $S$.

Since the output $S$ contains all vertices in $S^\ast$ and none from its complement, we have $S=S^\ast$. Thus, the DLL algorithm can indeed produce a minimum feedback vertex set.
\end{proof}

Next, we establish a deterministic upper bound on the size of the output set $S$, which holds for any random labeling and depends on the component structure of the graph.

\begin{prop}[Worst-case bound on FVS size]\label{prop:DL_bound}
Let $D$ be a digraph with $n$ vertices and $c$ connected components. For any labeling, the output $S$ of the DLL algorithm satisfies $|S|\le n-c$.
\end{prop}

\begin{proof}
We establish a lower bound on $|V(D)\setminus S|$ by inspecting the highest-labeled vertex in each component. Consider any connected component $W$. We show that at least one vertex of $W$ does not belong to $S$. Let $v_{\max}$ be the vertex in $W$ with the largest label assigned by $L$. By definition, $v_{\max}$ has no neighbors with a larger label, so it has no right neighbors at all. Therefore, $v_{\max} \notin S$. Since at least $c$ vertices are not in $S$, we deduce that $|S|\le n-c$. \end{proof}

\subsection{Recovering the AGJS bound}\label{subsec:dl_akbari} We now show that the probability a vertex is excluded by the DLL algorithm is precisely the term  $\rho_{D}(v)$ from the AGJS bound \eqref{eq:AGJS}.

\begin{prop}[DLL probability identity]\label{prop:prob_identity}
For any vertex $v\in V(D)$,
\[
\Prb(v\notin S)
\,=\, \frac{1}{1+d_v^+}+\frac{1}{1+d_v^-}-\frac{1}{1+d_v}
\,=\, \rho_D(v).
\]
\end{prop}

\begin{proof}
By the definition of the DLL algorithm, a vertex $v$ is \emph{not} included in $S$ if and only if it has no right in-neighbors or no right out-neighbors.

Let $E_v^-$ be the event that $v$ has no right in-neighbors (i.e., all vertices in $N^-(v)$ appear before $v$ in the random labeling $L$).
Let $E_v^+$ be the event that $v$ has no right out-neighbors (i.e., all vertices in $N^+(v)$ appear before $v$ in $L$).

We wish to compute $\Prb(v \notin S) = \Prb(E_v^- \cup E_v^+)$. By the principle of inclusion-exclusion, this is:
\[
\Prb(v \notin S) = \Prb(E_v^-) + \Prb(E_v^+) - \Prb(E_v^- \cap E_v^+).
\]
Each of these probabilities depends only on the relative ordering of $v$ and its neighbors. In a uniform random permutation:
\begin{itemize}
    \item $\Prb(E_v^-)$ is the probability that $v$ is the last element among the $d_v^- + 1$ vertices in $N^-(v) \cup \{v\}$. Thus, $\Prb(E_v^-) = \frac{1}{1+d_v^-}$.
    
    \item $\Prb(E_v^+)$ is the probability that $v$ is the last element among the $d_v^+ + 1$ vertices in $N^+(v) \cup \{v\}$. Thus, $\Prb(E_v^+) = \frac{1}{1+d_v^+}$.
    
    \item $\Prb(E_v^- \cap E_v^+)$ is the probability that $v$ is the last element among all $d_v + 1$ vertices in $N(v) \cup \{v\}$. Thus, $\Prb(E_v^- \cap E_v^+) = \frac{1}{1+d_v}$.
\end{itemize}
Substituting these values into the inclusion-exclusion formula gives the identity:
\[
\Prb(v\notin S) = \frac{1}{1+d_v^+}+\frac{1}{1+d_v^-}-\frac{1}{1+d_v} = \rho_D(v). \qedhere
\]
\end{proof}

\begin{cor}[Expected FVS size]\label{cor:DL_expectation}
The expected size of the set $S$ returned by the DLL algorithm is
\[ \E[|S|]= n - \sum_{v \in V(D)} \rho_D(v). \]
\end{cor}

\begin{proof}
By linearity of expectation and Proposition~\ref{prop:prob_identity}, we have:
\[ \E[|S|] = \sum_{v \in V(D)} \Prb(v \in S) = \sum_{v \in V(D)} (1 - \Prb(v \notin S)) = n - \sum_{v \in V(D)} \rho_D(v). \qedhere \]
\end{proof}

This result provides an alternative derivation of the AGJS bound  \eqref{eq:AGJS}. The expected FVS size from Corollary~\ref{cor:DL_expectation} gives an upper bound on the minimum FVS size, $\vv{\beta}(D)$. Using the identity $\vv{\alpha}(D) = n - \vv{\beta}(D)$, this translates directly to the familiar lower bound on the acyclic number:
\[ 
\vv{\alpha}(D) = n - \vv{\beta}(D) \ge n - \E[|S|] = \sum_{v\in V(D)} \rho_D(v). 
\]
Thus, the first-order bound \eqref{eq:AGJS} is recovered, which confirms that our FVS-based approach is correctly calibrated. This consistency sets the stage for the variance-based second-order refinement developed in the next section.

\section{The variance of the DLL algorithm}\label{sec:variance}

Recall that the output of the DLL algorithm is a random feedback vertex set $S$. The previous section computed a formula for the expectation of the random variable $|S|$. To develop our variance-based refinement, we move beyond the first moment of $|S|$ and derive an explicit formula for its variance. We begin by expressing $|S|$ as a sum of indicator random variables. To achieve this, for each vertex $v$, let $\mathbb{I}_v$ be the indicator for the event $v \in S$. For two vertices $u$ and $v$, we use $\Prb(u,v\in S)$ to mean $\Prb(u\in S \text{ and } v\in S)$.  The variance of $|S|=\sum_v \mathbb{I}_v$ then decomposes into a sum of variances and covariances:
\begin{align}\label{eq:variance-decomp}
\Var(|S|) &= \sum_{v}\Var(\mathbb{I}_v)+\sum_{u\ne v}\Cov(\mathbb{I}_u,\mathbb{I}_v) \nonumber \\
&= \sum_v \Prb(v\in S)\Prb(v\notin S) + \sum_{u\ne v}\Bigl(\Prb(u,v\in S)-\Prb(u\in S)\Prb(v\in S)\Bigr) \nonumber \\
&= \sum_v \rho_D(v)\bigl(1-\rho_D(v)\bigr) + \sum_{u\ne v}\Bigl(\Prb(u,v\in S)-(1-\rho_D(u))(1-\rho_D(v))\Bigr).
\end{align}
In the last step, we used the identity $\Prb(v\in S)=1-\rho_D(v)$ from Proposition~\ref{prop:prob_identity}. We already have an explicit formula for $\rho_D(v)$ from \eqref{eq:rho_D(v)}. It remains to derive a formula for the joint probability $\Prb(u,v\in S)$ for any pair of distinct vertices, which is the focus of the present section.

This section is organized as follows. We begin in \S\ref{subsec:pie_setup} by outlining our strategy, which involves analyzing a complementary event using the Principle of Inclusion-Exclusion. To evaluate the terms in this expansion, we establish a general ordering lemma in \S\ref{subsec:g_lemma}. We then apply this framework to three scenarios: nonadjacent vertices in \S\ref{subsec:nonadj}, adjacent vertices with one arc in \S\ref{subsec:adj:one-arc}, and adjacent vertices with two arcs in \S\ref{subsec:adj:two-arcs}. Finally, in Theorem~\ref{thm:cov-catalog}, we synthesize the analysis into a final theorem: a complete recipe for the covariance.

\subsection{A strategy via inclusion-exclusion}\label{subsec:pie_setup}
A direct calculation of $\Prb(u,v\in S)$ seems difficult. The condition for inclusion requires the \emph{existence} of both a right in-neighbor and a right out-neighbor; this structure is not easy to work with. It is more tractable to analyze the complementary probability, $\Prb(u \notin S \text{ or } v \notin S)$. The condition for a vertex $w$ to be excluded (that \emph{all} its in-neighbors or \emph{all} its out-neighbors are to the left of $w$) is defined by universal quantifiers. This structure is ideally suited for a direct application of the Principle of Inclusion-Exclusion.

To formalize this strategy, we consider the following events for a vertex $x$ based on its ordering relative to its neighbors:
\begin{align*}
    E_x^- &= \{\text{all in-neighbors of $x$ appear before $x$}\}, \\ 
E_x^+ &= \{\text{all out-neighbors of $x$ appear before $x$}\}.
\end{align*}
As a reminder, saying that all in-/out-neighbors appear before $x$ is the same as saying $x$ has no right in-/out-neighbors. As discussed in the proof of Proposition~\ref{prop:prob_identity}, the event $x \notin S$ is equivalent to the union $E_x^-\cup E_x^+$. For two distinct vertices $u,v$, we are interested in the probability:
\[
\Psi(u,v) \colonequals \Prb(u\notin S \text{ or } v\notin S) = \Prb\bigl( (E_u^-\cup E_u^+)\cup (E_v^-\cup E_v^+)\bigr).
\]
Let $A_1\colonequals E_u^-$, $A_2\colonequals E_u^+$, $A_3\colonequals E_v^-$, and $A_4\colonequals E_v^+$. Since $\Psi(u,v)=\Prb(A_1\cup A_2\cup A_3\cup A_4)$, we can apply the inclusion-exclusion formula to this union probability to express:
\begin{equation}\label{eq:PIE-4}
\Psi(u,v) = \Sigma_1 - \Sigma_2 + \Sigma_3 - \Sigma_4,
\end{equation}
where $\Sigma_k$ is the sum of probabilities of all $k$-fold intersections of these events:
\begin{align*}
\Sigma_1 &= \Prb(A_1)+\Prb(A_2)+\Prb(A_3)+\Prb(A_4) \\ 
\Sigma_2 &= \Prb(A_1\cap A_2)+\Prb(A_1\cap A_3)+\cdots + \Prb(A_3\cap A_4) \quad\quad  (6 \text{ terms})\\
\Sigma_3 &= \Prb(A_1\cap A_2\cap A_3) + \Prb(A_1\cap A_2\cap A_4) + \Prb(A_1\cap A_3\cap A_4) + \Prb(A_2\cap A_3\cap A_4) \\ 
\Sigma_4 &= \Prb(A_1\cap A_2\cap A_3\cap A_4).
\end{align*}
The desired joint probability is then $\Prb(u,v\in S)=1-\Psi(u,v)$. The evaluation of each term is sensitive to the relationship between $u$ and $v$. We therefore bifurcate our analysis into two scenarios: the nonadjacent case and the adjacent cases. To facilitate these calculations, we denote neighborhood overlaps as:
\begin{align*}
n_{uv}^{--} &\colonequals |N^-(u)\cap N^-(v)|, & n_{uv}^{++} &\colonequals |N^+(u)\cap N^+(v)|, & n_{uv}^{-+} &\colonequals |N^-(u)\cap N^+(v)|, \\
n_{uv}^{+-} &\colonequals |N^+(u)\cap N^-(v)|, & n_{uv}^{\pm -} &\colonequals |N(u)\cap N^-(v)|, & n_{uv}^{-\pm} &\colonequals |N^-(u)\cap N(v)|, \\
n_{uv}^{\pm +} &\colonequals |N(u)\cap N^+(v)|, & n_{uv}^{+\pm} &\colonequals |N^+(u)\cap N(v)|, & n_{uv}^{\pm \pm} &\colonequals |N(u)\cap N(v)|.
\end{align*}

\subsection{A general ordering lemma}\label{subsec:g_lemma}
Before proceeding to the case analysis, we present a lemma for computing the probability of joint ordering events. 

\begin{lem}\label{lem:g_function}
Let $X,Y$ be two sets of vertices, and let $v,w$ be two distinct vertices not in $X\cup Y$. In a uniform random permutation of the vertices in $Z\colonequals X\cup Y\cup\{v,w\}$, the probability that all vertices of $X$ appear before $v$ and all vertices of $Y$ appear before $w$ is given by
\[
g(X,Y)\;=\; \frac{1}{|X|+|Y|-|X\cap Y|+2} \left( \frac{1}{|X|+1} + \frac{1}{|Y|+1} \right).
\]
\end{lem}

\begin{proof}
Let us write $a \prec b$ to denote that vertex $a$ appears before vertex $b$ in a uniform random permutation of the vertices in $Z$. Let $x\colonequals |X|$, $y\colonequals |Y|$, $t\colonequals |X\cap Y|$, and let $N\colonequals |Z|=x+y-t+2$ be the total number of vertices. Denote by $X\prec v$ the event that all vertices of $X$ appear before $v$. We want to find the probability of the event $\mathcal{E}$, defined as $\{X \prec v \text{ and } Y \prec w\}$.

We condition on the last element, $z$, in the random permutation of $Z$. For $\mathcal{E}$ to occur, $z$ cannot be in $X$ (since $X \prec v$) or in $Y$ (since $Y \prec w$). By hypothesis, $v,w \notin X \cup Y$, so the only candidates for the last element are $v$ and $w$.

\textbf{Case 1:} The last element is $v$.

Any element of $Z$ is equally likely to be last, so $\Prb(z=v) = 1/N$. Conditional on $v$ being the last element, the condition $X \prec v$ is automatically satisfied. The event $\mathcal{E}$ thus reduces to the single condition $Y \prec w$. This depends only on the relative ordering of the $y+1$ vertices in $Y\cup\{w\}$; the event $Y \prec w$ holds if and only if $w$ is the last among them, which has probability $1/(y+1)$.

\textbf{Case 2:} The last element is $w$.

By a symmetric argument, $\Prb(z=w) = 1/N$. Conditional on $w$ being last, the condition $Y \prec w$ is satisfied. The event $\mathcal{E}$ now requires only $X \prec v$. This occurs if and only if $v$ is the last among the $x+1$ vertices in $X\cup\{v\}$, which has probability $1/(x+1)$.

Since these two cases are disjoint and are the only ways for $\mathcal{E}$ to occur, we sum their probabilities:
\begin{align*}
\Prb(\mathcal{E}) &= \Prb(z=v)\Prb(\mathcal{E}\mid z=v) + \Prb(z=w)\Prb(\mathcal{E}\mid z=w) \\
&= \frac{1}{N} \cdot \frac{1}{y+1} + \frac{1}{N} \cdot \frac{1}{x+1} = \frac{1}{x+y-t+2} \left( \frac{1}{x+1} + \frac{1}{y+1} \right). \qedhere
\end{align*}
\end{proof}

\subsection{Case 1: Nonadjacent vertices}\label{subsec:nonadj}
We begin with the simpler case where there is no arc between $u$ and $v$. The lack of a directed edge ensures that $u\notin N(v)$ and $v\notin N(u)$. Consequently, for any choice of neighborhoods $X \subseteq N(u)$ and $Y \subseteq N(v)$, the conditions of Lemma~\ref{lem:g_function} are met.

The four singleton probabilities for $\Sigma_1$ are given by the probability that a vertex is the last among its local neighborhood:
\begin{align*}
\Prb(E_u^-) &= \frac{1}{d_u^-+1}, & \Prb(E_v^-) &= \frac{1}{d_v^-+1}, \\
\Prb(E_u^+) &= \frac{1}{d_u^++1}, & \Prb(E_v^+) &= \frac{1}{d_v^++1}.
\end{align*}
The six pairwise intersection probabilities for $\Sigma_2$ consist of two self-pairs and four mixed-pairs. The self-pairs are:
\[
\Prb(E_u^-\cap E_u^+) = \frac{1}{d_u+1}, \qquad \Prb(E_v^-\cap E_v^+) = \frac{1}{d_v+1}.
\]
The four mixed-pairs are all direct applications of Lemma~\ref{lem:g_function}:
\begin{align*}
\Prb(E_u^-\cap E_v^-)&=g\bigl(N^-(u),N^-(v)\bigr), & \Prb(E_u^-\cap E_v^+)=g\bigl(N^-(u),N^+(v)\bigr),\\
\Prb(E_u^+\cap E_v^-)&=g\bigl(N^+(u),N^-(v)\bigr), & \Prb(E_u^+\cap E_v^+)=g\bigl(N^+(u),N^+(v)\bigr).
\end{align*}
Similarly, the four triple intersection probabilities for $\Sigma_3$ are given by:
\begin{align*}
\Prb(E_u^-\cap E_u^+\cap E_v^-)&=g\bigl(N(u), N^-(v)\bigr), & \Prb(E_u^-\cap E_u^+\cap E_v^+)=g\bigl(N(u), N^+(v)\bigr), \\
\Prb(E_u^-\cap E_v^-\cap E_v^+)&=g\bigl(N^-(u), N(v)\bigr), & \Prb(E_u^+\cap E_v^-\cap E_v^+)=g\bigl(N^+(u), N(v)\bigr).
\end{align*}
Finally, the single quadruple probability for $\Sigma_4$ is:
\[
\Prb(E_u^-\cap E_u^+\cap E_v^-\cap E_v^+) = g\bigl(N(u),N(v)\bigr).
\]
Substituting these 15 quantities into the inclusion-exclusion formula \eqref{eq:PIE-4} yields a closed-form expression for $\Prb(u,v\in S)$ in the nonadjacent case.

\subsection{Case 2: Adjacent vertices with one arc}\label{subsec:adj:one-arc}
Suppose we have $u\to v$ but $v\nrightarrow u$ (the case $v\to u$ and $u\nrightarrow v$ follows by symmetry). This implies $v \in N^+(u)$ and $u \in N^-(v)$, creating a dependency that invalidates the direct application of Lemma~\ref{lem:g_function} for some terms and makes certain joint events impossible. In particular, the arc $u\to v$ forces one pairwise intersection to be impossible.

\begin{claim}[An impossible intersection]\label{cl:impossible}
If $u\to v$ is an arc, the joint event $E_u^+\cap E_v^-$ cannot occur, so $\Prb(E_u^+\cap E_v^-)=0$.
\end{claim}
\begin{proof}
The event $E_u^+$ requires all out-neighbors of $u$ to appear before $u$. Since $v \in N^+(u)$, this implies the relative order $v \prec u$. The event $E_v^-$ requires all in-neighbors of $v$ to appear before $v$. Since $u \in N^-(v)$, this implies $u \prec v$. These required orderings are mutually exclusive, so the joint event has probability zero.
\end{proof}

Claim~\ref{cl:impossible} has a cascading effect on the expansion \eqref{eq:PIE-4}, as the term $\Prb(E_u^+\cap E_v^-)$ and any higher-order term (in $\Sigma_3$ and $\Sigma_4$) containing this intersection will vanish. We now evaluate the remaining nonzero terms again under the assumption that $u\to v$ and $v\nrightarrow u$. The singleton probabilities for $\Sigma_1$ and the self-pair probabilities in $\Sigma_2$ are unaffected by the arc. We still have:
\[
\Prb(E_u^-\cap E_u^+) = \frac{1}{d_u+1}, \qquad \Prb(E_v^-\cap E_v^+) = \frac{1}{d_v+1}.
\]
Of the mixed-pairs, the term $\Prb(E_u^-\cap E_v^+)$ can still be calculated with Lemma~\ref{lem:g_function}:
\[
\Prb(E_u^-\cap E_v^+) = g(N^-(u), N^+(v)).
\]
The remaining two pairwise intersections, however, require a modified analysis.

\begin{claim}[Special pairwise intersections]\label{cl:keypairs}
Assume $u\to v$. The joint probabilities are:
\begin{align*}
\Prb(E_u^-\cap E_v^-) &= \frac{1}{\bigl(d_u^- + d_v^- - n_{uv}^{--} + 1\bigr)\,\bigl(d_u^-+1\bigr)}, \\
\Prb(E_u^+\cap E_v^+) &= \frac{1}{\bigl(d_u^+ + d_v^+ - n_{uv}^{++} + 1\bigr)\,\bigl(d_v^+ + 1\bigr)}.
\end{align*}
\end{claim}
\begin{proof}
We prove the result for $\Prb(E_u^-\cap E_v^-)$; the proof for $\Prb(E_u^+\cap E_v^+)$ is symmetric.

Let $X=N^-(u)$ and $Y=N^-(v)$. We are interested in the event $\mathcal{E}$ where all of $X$ appear before $u$ (event $X \prec u$) and all of $Y$ appear before $v$ (event $Y \prec v$). Since there is an arc $u\to v$, we have $u \in N^-(v)=Y$. The condition $Y \prec v$ therefore forces the relative ordering $u \prec v$.

The outcome of the event $\mathcal{E}$ depends only on the relative ordering of the vertices in the set $Z \colonequals X \cup Y \cup \{v\}$. Note that $u \in Y$, so $u \in Z$. The size of this set is:
\[
|Z| = |X \cup Y \cup \{v\}| = |X \cup Y| + 1 = |X|+|Y|-|X \cap Y|+1 = d_u^- + d_v^- - n_{uv}^{--} + 1.
\]
For the event $\mathcal{E}$ to occur, two conditions must be met:
\begin{enumerate}
    \item The vertex $v$ must appear after all vertices in $Y$. Since $u \in Y$, $v$ must appear after $u$.
    \item The vertex $u$ must appear after all vertices in $X$.
\end{enumerate}
Combining these two conditions implies that $v$ must appear after all of $X\cup Y$, that is, $v$ must be the last element in the random permutation of all vertices in $Z$.

We formalize the calculation using conditional probability. Let $E_1$ be the event that $v$ is the last element in a permutation of $Z$. The probability of this event is:
\[
\Prb(E_1) = \frac{1}{|Z|} = \frac{1}{d_u^- + d_v^- - n_{uv}^{--} + 1}.
\]
Let us now consider the conditional probability $\Prb(\mathcal{E} \mid E_1)$. Conditioned on $v$ being last, the requirement $Y \prec v$ is automatically satisfied. The event $\mathcal{E}$ now only requires $X \prec u$. This remaining condition only depends on the relative ordering of the vertices in $X \cup \{u\}$. The probability that $u$ is the last among these $|X|+1$ vertices is:
\[
\Prb(\mathcal{E} \mid E_1) = \Prb(X \prec u) = \frac{1}{|X|+1} = \frac{1}{d_u^-+1}.
\]
By the law of conditional probabilities, the joint probability is the product:
\begin{align*}
\Prb(E_u^-\cap E_v^-) &= \Prb(\mathcal{E}) = \Prb(E_1) \cdot \Prb(\mathcal{E} \mid E_1) \\
&= \frac{1}{|Z|} \cdot \frac{1}{|X|+1} = \frac{1}{\bigl(d_u^- + d_v^- - n_{uv}^{--} + 1\bigr)\,\bigl(d_u^-+1\bigr)}.
\end{align*}
A symmetric argument holds for $\Prb(E_u^+\cap E_v^+)$, where the dependency $v \in N^+(u)$ forces $u$ to be the last element in the relevant permutation.
\end{proof}

Finally, we consider the higher-order terms. The impossible intersection from Claim~\ref{cl:impossible} implies that certain events have zero probability:
\begin{align*}
\Prb(E_u^-\cap E_u^+\cap E_v^-) =0, \quad \Prb(E_u^+\cap E_v^-\cap E_v^+) =0, \quad \Prb(E_u^-\cap E_u^+\cap E_v^-\cap E_v^+) =0.
\end{align*}
The remaining two triple terms are $\Prb(E_u^-\cap E_v^-\cap E_v^+)$ and $\Prb(E_u^-\cap E_u^+\cap E_v^+)$, which require a custom argument.

\begin{claim}[Special triple intersections]\label{cl:triple_uuvplus}
Assume $u\to v$. The two nonzero triple intersection probabilities are given below.
\begin{align*}
\Prb\bigl(E_u^-\cap E_u^+\cap E_v^+\bigr) &= \frac{1}{\bigl(d_u+d_v^+-n_{uv}^{\pm+}+1\bigr)\,(d_v^++1)} \\ 
\Prb\bigl(E_u^-\cap E_v^-\cap E_v^+\bigr) &= \frac{1}{\bigl(d_u^-+d_v-n_{uv}^{-\pm}+1\bigr)\,(d_u^-+1)}
\end{align*} 
\end{claim}

\begin{proof}
We explain the argument for the first identity; the proof for the second identity is similar. Let $X=N(u)$ and $Y=N^+(v)$. We compute $\Prb(\mathcal{E})$ where $\mathcal{E}$ is the event $\{X \prec u \text{ and } Y \prec v\}$. Since $u\to v$, the membership $v \in X$ forces the ordering $v \prec u$. Combined, these conditions imply that $u$ must be the last element in any permutation of the relevant set $Z \colonequals X \cup Y \cup \{u\}$.

The probability of the event $\mathcal{E}$ is therefore the probability that $u$ is last in $Z$, multiplied by the conditional probability that $Y \prec v$ (given $u$ is last in $Z$):
\begin{align*}
\Prb(\mathcal{E}) &= \Prb(u \text{ is last in } Z) \cdot \Prb(Y \prec v \mid u \text{ is last in } Z) \\
&= \frac{1}{|X \cup Y|+1} \cdot \frac{1}{|Y|+1} \\
&= \frac{1}{(d_u + d_v^+ - n_{uv}^{\pm+} + 1)(d_v^+ + 1)}. \qedhere
\end{align*}
\end{proof}

\subsection{Case 3: Adjacent vertices with two arcs}\label{subsec:adj:two-arcs} We now analyze the case where a directed 2-cycle exists between $u$ and $v$, meaning both arcs $u \to v$ and $v \to u$ are present.\footnote{The presence of such 2-cycles is the key distinction between a simple digraph and an \emph{oriented graph}, which forbids them. As oriented graphs are a special case of the digraphs considered here, all our results apply to them as well.}

We proceed to the analysis of the terms in the inclusion-exclusion expansion~\eqref{eq:PIE-4}. The four singleton probabilities and the two specific probabilities $\Prb(E_u^-\cap E_u^+)$ and $\Prb(E_v^-\cap E_v^+)$ are unaffected, as their calculations are local to a single vertex.

However, we show that all the other terms in the expansion \eqref{eq:PIE-4} vanish. Since the directed edges $u\to v$ and $v\to u$ both exist, the four primary events imply an ordering relation between $u$ and $v$:
\begin{align*}
E_u^- \text{ requires } v \prec u, && E_v^- \text{ requires } u \prec v, \\
E_u^+ \text{ requires } v \prec u, && E_v^+ \text{ requires } u \prec v.
\end{align*}
Since $u\prec v$ and $v\prec u$ are mutually exclusive, any event that requires both to be true is impossible: 
\[
\Prb(E_u^-\cap E_v^-)=\Prb(E_u^-\cap E_v^+)=\Prb(E_u^+\cap E_v^-)=\Prb(E_u^+\cap E_v^+)=0.
\]
Furthermore, every triple and quadruple intersection is a sub-event of at least one of the impossible mixed pairs identified above. For example, the event $E_u^-\cap E_u^+\cap E_v^-$ requires the event $E_u^+\cap E_v^-$ to hold. Therefore, all triple and quadruple intersection probabilities are $0$.

We now consolidate our findings into a single theorem. This result provides a complete recipe for calculating the covariance by cataloging the 15 intersection probabilities required for the inclusion-exclusion formula in computing $\Psi(u,v)$.

\begin{thm}[Covariance formula]\label{thm:cov-catalog}
For any two distinct vertices $u,v \in V(D)$, the covariance of their indicator variables $\mathbb{I}_u$ and $\mathbb{I}_v$ (for the events $u \in S$ and $v \in S$) is given by:
\[
\Cov(\mathbb{I}_u,\mathbb{I}_v)=\Prb(u,v\in S)-(1-\rho_D(u))(1-\rho_D(v)),
\]
where $\Prb(u,v\in S)=1-\Psi(u,v)$. The term $\Psi(u,v)$ is computed using the inclusion-exclusion formula \eqref{eq:PIE-4}, with the required probabilities given in Table~\ref{tab:complete-catalog}.
\end{thm}

\begin{table}[htbp]
\centering
\caption{Complete catalog of terms in \eqref{eq:PIE-4} for all cases of vertex pairs.}
\label{tab:complete-catalog}
\small
\begin{tabular}{@{}llll@{}}
\toprule
\textbf{Term} & \textbf{Case 1: Nonadjacent} & \textbf{Case 2: $u \to v$} & \textbf{Case 3: $u \curvyarrows v$} \\ \midrule
\multicolumn{4}{l}{\textit{\textbf{$\Sigma_1$ Terms}}} \\
$\Prb(E_u^-)$ & $\frac{1}{d_u^-+1}$ & $\frac{1}{d_u^-+1}$ & $\frac{1}{d_u^-+1}$ \\[.7em]
$\Prb(E_u^+)$ & $\frac{1}{d_u^++1}$ & $\frac{1}{d_u^++1}$ & $\frac{1}{d_u^++1}$ \\[.7em]
$\Prb(E_v^-)$ & $\frac{1}{d_v^-+1}$ & $\frac{1}{d_v^-+1}$ & $\frac{1}{d_v^-+1}$ \\[.7em]
$\Prb(E_v^+)$ & $\frac{1}{d_v^++1}$ & $\frac{1}{d_v^++1}$ & $\frac{1}{d_v^++1}$ \\ \midrule
\multicolumn{4}{l}{\textit{\textbf{$\Sigma_2$ Terms}}} \\
$\Prb(E_u^-\cap E_u^+)$ & $\frac{1}{d_u+1}$ & $\frac{1}{d_u+1}$ & $\frac{1}{d_u+1}$ \\[.7em]
$\Prb(E_v^-\cap E_v^+)$ & $\frac{1}{d_v+1}$ & $\frac{1}{d_v+1}$ & $\frac{1}{d_v+1}$ \\[.7em]
$\Prb(E_u^-\cap E_v^-)$ & $\frac{1}{d_u^-+d_v^--n_{uv}^{--}+2}\left(\frac{1}{d_u^-+1}+\frac{1}{d_v^-+1}\right)$ & $\frac{1}{(d_u^-+d_v^--n_{uv}^{--}+1)(d_u^-+1)}$ & $0$ \\[.7em]
$\Prb(E_u^-\cap E_v^+)$ & $\frac{1}{d_u^-+d_v^+-n_{uv}^{-+}+2}\left(\frac{1}{d_u^-+1}+\frac{1}{d_v^++1}\right)$ & $\frac{1}{d_u^-+d_v^+-n_{uv}^{-+}+2}\left(\frac{1}{d_u^-+1}+\frac{1}{d_v^++1}\right)$ & $0$ \\[.7em]
$\Prb(E_u^+\cap E_v^-)$ & $\frac{1}{d_u^++d_v^--n_{uv}^{+-}+2}\left(\frac{1}{d_u^++1}+\frac{1}{d_v^-+1}\right)$ & $0$ & $0$ \\[.7em]
$\Prb(E_u^+\cap E_v^+)$ & $\frac{1}{d_u^++d_v^+-n_{uv}^{++}+2}\left(\frac{1}{d_u^++1}+\frac{1}{d_v^++1}\right)$ & $\frac{1}{(d_u^++d_v^+-n_{uv}^{++}+1)(d_v^++1)}$ & $0$ \\ \midrule
\multicolumn{4}{l}{\textit{\textbf{$\Sigma_3$ Terms}}} \\
$\Prb(E_u^-\cap E_u^+\cap E_v^-)$ & $\frac{1}{d_u+d_v^--n_{uv}^{\pm-}+2}\left(\frac{1}{d_u+1}+\frac{1}{d_v^-+1}\right)$ & $0$ & $0$ \\[.7em]
$\Prb(E_u^-\cap E_u^+\cap E_v^+)$ & $\frac{1}{d_u+d_v^+-n_{uv}^{\pm+}+2}\left(\frac{1}{d_u+1}+\frac{1}{d_v^++1}\right)$ & $\frac{1}{(d_u+d_v^+-n_{uv}^{\pm+}+1)(d_v^++1)}$ & $0$ \\[.7em]
$\Prb(E_u^-\cap E_v^-\cap E_v^+)$ & $\frac{1}{d_u^-+d_v-n_{uv}^{-\pm}+2}\left(\frac{1}{d_u^-+1}+\frac{1}{d_v+1}\right)$ & $\frac{1}{(d_u^-+d_v-n_{uv}^{-\pm}+1)(d_u^-+1)}$ & $0$ \\[.7em]
$\Prb(E_u^+\cap E_v^-\cap E_v^+)$ & $\frac{1}{d_u^++d_v-n_{uv}^{+\pm}+2}\left(\frac{1}{d_u^++1}+\frac{1}{d_v+1}\right)$ & $0$ & $0$ \\ \midrule
\multicolumn{4}{l}{\textit{\textbf{$\Sigma_4$ Term}}} \\
$\Prb(E_u^-\cap E_u^+\cap E_v^-\cap E_v^+)$ & $\frac{1}{d_u+d_v-n_{uv}^{\pm\pm}+2}\left(\frac{1}{d_u+1}+\frac{1}{d_v+1}\right)$ & $0$ & $0$ \\ \bottomrule
\end{tabular}
\end{table}

\section{The variance-based lower bound}\label{sec:bd}

We have now assembled all the necessary components for our second main result: a variance-based refinement of the AGJS bound \eqref{eq:AGJS}. The previous sections provided a randomized algorithm for generating a feedback vertex set, $S$, and established a formula for calculating the variance of its size, $\Var(|S|)$. In this section, we connect these pieces using the Bhatia--Davis inequality \cite{bhatia2000} to produce a refined lower bound on $\vv{\alpha}(D)$.

\begin{thm}\label{thm:main_acl}
Let $D$ be a digraph with $n$ vertices and $c$ connected components. Let $S$ be the random feedback vertex set produced by the DLL algorithm. Then 
\begin{equation}\label{eq:ACL-final-simple}
\vv{\alpha}(D)
\;\ge\; \sum_{v\in V(D)} \rho_D(v) \;+\; \frac{\Var(|S|)}{\left(\sum_{v\in V(D)} \rho_D(v)\right)-c}.
\end{equation}
If the denominator is zero, the second term is taken to be zero.
\end{thm}

\begin{proof}
We apply the Bhatia--Davis inequality, which bounds the variance of a random variable by its expectation and range. Let $|S|$ be the size of the random FVS generated by the DLL algorithm. From our previous analysis, we have three key ingredients:
\begin{itemize}
    \item The minimum FVS size provides a lower bound, so $|S|\ge \vv{\beta}(D)$.
    \item Proposition~\ref{prop:DL_bound} gives a deterministic upper bound, $|S| \leq n-c$.
    \item Proposition~\ref{prop:prob_identity} and Corollary~\ref{cor:DL_expectation} give the expected size, $\mu \colonequals \E[|S|]=n-\sum_v\rho_D(v)$.
\end{itemize}
For a random variable $X$ with range $[m,M]$, the Bhatia--Davis inequality \cite{bhatia2000} states that $\Var(X) \le (M-\E[X])(\E[X]-m)$. Applying this to $X=|S|$ with $m=\vv{\beta}(D)$ and $M=n-c$, we get:
\[
\Var(|S|) \le (M-\mu)(\mu-\vv{\beta}(D)).
\]
If the term $M-\mu$ is positive, we can rearrange this inequality to get an improved upper bound on the minimum FVS size, $\vv{\beta}(D)$:
\[
\vv{\beta}(D)\;\le\; \mu - \frac{\Var(|S|)}{M-\mu}.
\]
We now use the identity $\vv{\alpha}(D)=n-\vv{\beta}(D)$ to translate this into a lower bound on $\vv{\alpha}(D)$. Substituting the expressions for $\mu$ and $M$, we obtain the desired inequality:
\begin{align*}
\vv{\alpha}(D) &= n-\vv{\beta}(D) \ge n - \left(\mu - \frac{\Var(|S|)}{M-\mu}\right) = (n-\mu) + \frac{\Var(|S|)}{M-\mu} \\
&= \left(n - \left(n - \sum_{v\in V(D)} \rho_D(v)\right)\right) + \frac{\Var(|S|)}{\left(n-c\right) - \left(n-\sum_v \rho_D(v)\right)} \\
&= \sum_{v\in V(D)} \rho_D(v) + \frac{\Var(|S|)}{\sum_v \rho_D(v) - c}.
\end{align*}
If the term $M-\mu=0$, then the denominator $\sum_v \rho_D(v) - c$ vanishes; in this case, the correction term is taken to be zero. This completes the proof.
\end{proof}

Next, we address the case when the denominator $\sum_{v} \rho_{D}(v)-c$ is zero.

\begin{lem}\label{lem:exceptions}
Let $D$ be a digraph with $c$ connected components. Then $\sum_{v} \rho_{D}(v)-c=0$ if and only if every connected component of $D$ is a complete symmetric digraph (namely, a graph which has an arc between every ordered pair of vertices).
\end{lem}

\begin{proof}
Let $S$ denote the random FVS produced as an output of the DLL algorithm.  We have $\E[|S|]= n - \sum_{v \in V(D)} \rho_D(v)$ from Corollary~\ref{cor:DL_expectation}. Therefore, the equality $\sum_{v} \rho_{D}(v)-c=0$ holds if and only if $\E[|S|] = n-c$. By Proposition~\ref{prop:DL_bound}, the maximum possible value of $|S|$ is also $n-c$. Since the average value satisfies $\E[|S|] = n-c$, the random variable $|S|$ must be constant, namely $n-c$. It follows that in \emph{every} ordering of the vertices, the DLL algorithm produces a FVS of size exactly $n-c$. 

From the proof of Proposition~\ref{prop:DL_bound}, each connected component of $D$ has at least one vertex not in $S$. We have exactly $c$ elements not in $S$, which forces exactly one element from each connected component not to be in $S$ for every ordering of the vertices. This latter condition forces each connected component to be a complete symmetric digraph. Indeed, if a connected component were not a complete symmetric digraph, there would exist two vertices $u$ and $v$ such that (i) $u$ and $v$ are nonadjacent, (ii) $u\to v$ but $v\nrightarrow u$, or (iii) $v\to u$ but $u\nrightarrow v$. In all of these cases, the vertex order in which $u$ and $v$ appear as the last two vertices will produce a feedback vertex set not containing $u$ and $v$. This is because, in such an ordering, neither $u$ nor $v$ has both a right in-neighbor \emph{and} a right out-neighbor; so, they do not belong to the feedback vertex set produced by the DLL algorithm.
\end{proof}

\begin{rem}[Interpreting the bound]\label{rem:denominator}
The denominator of the refinement term in equation~\eqref{eq:ACL-final-simple} is always nonnegative. Indeed, $\sum_{v\in V(D)} \rho_D(v) - c$ is the gap between the expected number of vertices not in $S$ and the minimum possible number of vertices not in $S$. Our bound therefore provides an improvement over the AGJS bound \eqref{eq:AGJS} precisely when both the variance $\Var(|S|)$ and this denominator are strictly positive. The denominator is zero if and only if $D$ is a complete symmetric digraph. In this case, the maximum output size from the DLL algorithm is equal to its expected size. This rare scenario is characterized by Lemma~\ref{lem:exceptions}. In these exceptional cases, we set the correction term to zero and our bound coincides with the AGJS bound~\eqref{eq:AGJS}.
\end{rem}

In the important case of connected graphs, the bound simplifies as follows:

\begin{cor}\label{cor:connected_bound}
If a digraph $D$ is connected, then
\[
\vv{\alpha}(D) \;\ge\; \sum_{v\in V(D)}\left(\frac{1}{1+d_v^+}+\frac{1}{1+d_v^-}-\frac{1}{1+d_v}\right)+\frac{\Var(|S|)}{\left(\sum_{v \in V(D)}\left(\frac{1}{1+d_v^+}+\frac{1}{1+d_v^-}-\frac{1}{1+d_v}\right)\right)-1}.
\]
\end{cor}
\begin{proof}
This follows directly by substituting the value $c=1$ for a connected graph into the denominator term in \eqref{eq:ACL-final-simple} and substituting the explicit formula for $\rho_D(v)$ from \eqref{eq:rho_D(v)}.
\end{proof}

\section{Numerical Comparisons on Random Graph Models}\label{sec:experiments}

We established two distinct refinements of the AGJS bound \eqref{eq:AGJS} for the maximum size of an induced acyclic subgraph in a digraph: a neighborhood-based method (Theorem~\ref{thm:main_selkow}) and a variance-based approach (Theorem~\ref{thm:main_acl}). Because their formulas are complex, an analytic comparison is unwieldy; so, we analyze each bound separately and then benchmark them on random models.

\subsection{Neighborhood-based bound}\label{subsect:num-neighbor}
From \eqref{eq:selkow-digraph}, the improvement for the neighborhood-based bound is:
\begin{align*}
\Delta_{\text{neigh}} &\colonequals \sum_{v\in V(D)} \rho_D(v)\cdot \max\left\{0,\;1-\rho_D(v)-\sum_{u\in N(v)}\rho_D(u)\right\}
\end{align*}
where $\rho_D(v)$ is given by the formula \eqref{eq:rho_D(v)}:
\[
\rho_D(v) = \frac{1}{1+d_v^{+}} + \frac{1}{1+d_v^{-}} - \frac{1}{1+d_v}.
\]
A vertex $v$ contributes to the improvement $\Delta_{\text{neigh}}$  only if $1-\rho_D(v)-\sum_{u\in N(v)}\rho_D(u)>0$. In that case, its contribution is $C_v=\rho_D(v)(1-\rho_D(v)-\sum_{u\in N(v)}\rho_D(u)).$ This is maximized (for fixed $v$) when $\sum_{u\in N(v)}\rho_D(u)$ is small, in which case $C_v\approx \rho_D(v)\bigl(1-\rho_D(v)\bigr)$, whose maximum occurs at $\rho_D(v)=\frac{1}{2}$. Vertices with $\rho_D(v)\approx\frac{1}{2}$ typically have relatively low in- and out-degrees. We conclude that the vertices that contribute to an improved neighborhood-based bound are those that (a) have relatively low in- and out-degrees, and (b) have neighbors with relatively high in- and out-degrees.

\subsection{Variance-based bound}\label{subsect:num-var}
From \eqref{eq:ACL-final-simple}, the improvement for the variance-based bound is:

\begin{align*}
\Delta_{\text{var}} &\colonequals \frac{\Var(|S|)}{\sum_{v\in V(D)} \rho_D(v) - c},
\end{align*}
where $S$ is the DLL output (Section~\ref{subsec:dl_definition}) and $c$ is the number of connected components.

From this formula, we conclude that a meaningful improvement of the variance-based bound over the AGJS bound~\eqref{eq:AGJS} is dependent on a sufficiently large $\Var(|S|)$.  Recall that the DLL algorithm produces a feedback vertex set for each of the $n!$ distinct vertex orderings using the rule: 
\begin{quote}
Let $\pi$ be a random vertex ordering from digraph $D$, and let $v_i$ be the vertex in the $i$-th position within $\pi$.  Include $v_i$ in $S$ if and only if $v_i$ has at least one right in-neighbor and at least one right out-neighbor.  
\end{quote}
Recall $\Var(|S|)$ measures the spread of $|S|$ over all $n!$ vertex orderings. To identify the type of digraph where $\Var(|S|)$ would be meaningful, we need to consider what graph properties would result in significant differences in FVS sizes as we vary the vertex orderings. 

Consider a graph with two classes: high-degree and low-degree vertices. High-degree vertices are likely to enter $S$ regardless of position, while low-degree vertices are sensitive to position. Orderings that place many low-degree vertices early tend to yield larger $|S|$, and the reverse yields smaller $|S|$, inflating $\Var(|S|)$.

The heuristic above suggests that heterogeneous degree structure should favor both refinements over \eqref{eq:AGJS}. We now test this on random models. We also compare the two refinements to each other.

\subsection{An Erd\H{o}s--R\'enyi random digraph model} In this model, each ordered pair $(u,v)$ is included as an arc independently with probability $p$ (both directions allowed). Degrees are nearly uniform in expectation, so we do not expect substantial gains from either refinement. Computer simulations support this prediction. As shown in Table~\ref{tab:erdos-renyi}, the two new lower bounds had no meaningful improvement over the AGJS bound~\eqref{eq:AGJS} in the random Erd\H{o}s--R\'enyi simple digraphs.

\begin{table}[ht]
\centering
\begin{tabular}{|>{\centering\arraybackslash}p{1.5cm}|
                >{\centering\arraybackslash}p{1.5cm}|
                >{\centering\arraybackslash}p{1.5cm}|
                >{\centering\arraybackslash}p{1.5cm}|>{\centering\arraybackslash}p{1.5cm}|>{\centering\arraybackslash}p{1.5cm}|>{\centering\arraybackslash}p{1.5cm}|>{\centering\arraybackslash}p{1.5cm}|}
\hline
\multicolumn{2}{|c|}{$p=0.05$} & \multicolumn{2}{|c|}{$p=0.50$} & \multicolumn{2}{|c|}{$p=0.95$} \\ \hline
$\Delta_{\text{neigh}}$ &  $\Delta_{\text{var}}$ & $\Delta_{\text{neigh}}$ &  $\Delta_{\text{var}}$ & $\Delta_{\text{neigh}}$ &  $\Delta_{\text{var}}$ \\  \hline
 0.0 & 0.31 & 0.0 & 0.50 & 0.0 & 0.91 \\ \hline
\end{tabular}
\caption{Erd\H{o}s--R\'enyi $G(n,p)$ digraphs with $n=100$; for each $p$ we generate $1000$ graphs, and the entries report the average improvements $\Delta_{\text{neigh}}$ and $\Delta_{\text{var}}$.}\label{tab:erdos-renyi}
\end{table}

\subsection{A two-type random digraph model}
We next use a graph model with two types of vertices: those with relatively high degrees and those with relatively low degrees. 

\begin{defn}[Two-type random digraph]
A random digraph $D$ is generated from the following parameters:
\begin{itemize}[leftmargin=2em]
  \item the number of vertices $n$;
  \item the proportion of low-degree vertices $p_{\text{low}}\in[0,1]$;
  \item arc probabilities $q_1$ (high--high), $q_2$ (high--low), $q_3$ (low--low), where $q_i\in [0, 1]$.
\end{itemize}
The generation process is as follows:
\begin{enumerate}
    \item The vertex set $V$ is partitioned into two disjoint subsets: $V_{\text{low}}$ of size $\lfloor p_{\text{low}} \cdot n \rfloor$ and $V_{\text{high}}$ of size $n - |V_{\text{low}}|$.
    
    \item For each ordered pair of distinct vertices $(u, v)$, there is the potential of zero, one, or two (directionally different) arcs between them. The probability that $u \rightarrow v$ exists is given by:
    \[
    p_{\text{edge}} = 
    \begin{cases} 
        q_1 & \text{if } u, v \in V_{\text{high}} \\
        q_2 & \text{if one vertex is in } V_{\text{high}} \text{ and the other is in } V_{\text{low}} \\
        q_3 & \text{if } u, v \in V_{\text{low}}
    \end{cases}
    \]
\end{enumerate}
\end{defn}

Based on the computer experiments, we observed that the relative strength of the two bounds is sensitive to the value of $q_3$. For instance, consider a scenario with a large proportion of low-degree type vertices ($p_{\text{low}} = 0.90$) and high connectivity involving the high-degree vertices ($q_1=0.70, q_2=0.50$). We observed that:
\begin{itemize}
    \item When the low-degree vertices are very sparsely connected (e.g., $q_3 = 0.01$), the neighborhood-based refinement appears to provide a stronger correction ($\Delta_{\text{neigh}} > \Delta_{\text{var}}$).
    \item However, by only slightly increasing the connectivity between low-degree vertices (e.g., to $q_3 = 0.05$), the variance-based bound becomes superior ($\Delta_{\text{var}} > \Delta_{\text{neigh}}$).
\end{itemize}
We could vary other variables (say $q_1$) while fixing the rest. The variation of $q_3$ was the most sensitive, which is why we focus only on this scenario.

Tables \ref{tab:two-type-1}, \ref{tab:two-type-2}, and \ref{tab:two-type-3} report the average value of improvements in a random simulation of the two-type digraph model run with the indicated parameter values.

\medskip \medskip 

\begin{table}[ht]
\centering
\begin{tabular}{|>{\centering\arraybackslash}p{1.5cm}|
                >{\centering\arraybackslash}p{1.5cm}|
                >{\centering\arraybackslash}p{1.5cm}|
                >{\centering\arraybackslash}p{1.5cm}|>{\centering\arraybackslash}p{1.5cm}|>{\centering\arraybackslash}p{1.5cm}|>{\centering\arraybackslash}p{1.5cm}|>{\centering\arraybackslash}p{1.5cm}|}
\hline
$q_3 \rightarrow$ & \multicolumn{3}{|c|}{$0.005$} & \multicolumn{3}{c|}{$0.01$} \\ \hline
$n \downarrow$ & \hyperref[eq:AGJS]{AGJS} & $\Delta_{\text{neigh}}$ &  $\Delta_{\text{var}}$ & \hyperref[eq:AGJS]{AGJS} & $\Delta_{\text{neigh}}$ &  $\Delta_{\text{var}}$ \\ \hline
100 & 20.89 & 8.78 & 3.84 & 19.93 & 5.82 & 3.53 \\ \hline
150 & 21.72 & 10.60 & 4.58  & 20.40 & 6.88 & 4.20 \\ \hline
200 & 21.94 & 11.39 & 5.02 & 20.67 & 7.32 & 4.52 \\ \hline
\end{tabular}
\caption{Two-type random digraphs with $p_{\text{low}}=0.90$, $q_1=0.70$, $q_2=0.50$. Columns list $q_3\in\{0.005,0.01\}$. Averages based on $100$ graphs for each $(n,q_3)$.}\label{tab:two-type-1}
\end{table}

\begin{table}[ht]
\centering
\begin{tabular}{|>{\centering\arraybackslash}p{1.5cm}|
                >{\centering\arraybackslash}p{1.5cm}|
                >{\centering\arraybackslash}p{1.5cm}|
                >{\centering\arraybackslash}p{1.5cm}|>{\centering\arraybackslash}p{1.5cm}|>{\centering\arraybackslash}p{1.5cm}|>{\centering\arraybackslash}p{1.5cm}|>{\centering\arraybackslash}p{1.5cm}|}
\hline
$q_3 \rightarrow$ & \multicolumn{3}{|c|}{$0.02$} & \multicolumn{3}{|c|}{$0.025$} \\ \hline
$n \downarrow$ & \hyperref[eq:AGJS]{AGJS} & $\Delta_{\text{neigh}}$ &  $\Delta_{\text{var}}$ & \hyperref[eq:AGJS]{AGJS} & $\Delta_{\text{neigh}}$ &  $\Delta_{\text{var}}$ \\ \hline
100 & 17.92 & 2.48 & 3.07 & 16.95 & 1.53 & 2.84 \\ \hline
150 & 18.12 & 2.52 & 3.50  & 17.14 & 1.49 & 3.20 \\ \hline
200 & 18.31 & 2.53 & 3.74 & 17.32 & 1.37 & 3.39 \\ \hline
\end{tabular}
\caption{Two-type random digraphs with $p_{\text{low}}=0.90$, $q_1=0.70$, $q_2=0.50$. Columns list $q_3\in\{0.02,0.025\}$. Averages based on $100$ graphs for each $(n,q_3)$.}\label{tab:two-type-2}
\end{table}

\begin{table}[ht]
\centering
\begin{tabular}{|>{\centering\arraybackslash}p{1.5cm}|
                >{\centering\arraybackslash}p{1.5cm}|
                >{\centering\arraybackslash}p{1.5cm}|
                >{\centering\arraybackslash}p{1.5cm}|>{\centering\arraybackslash}p{1.5cm}|>{\centering\arraybackslash}p{1.5cm}|>{\centering\arraybackslash}p{1.5cm}|>{\centering\arraybackslash}p{1.5cm}|}
\hline
$q_3 \rightarrow$ & \multicolumn{3}{|c|}{$0.05$} & \multicolumn{3}{|c|}{$0.10$} \\ \hline
$n \downarrow$ & \hyperref[eq:AGJS]{AGJS} & $\Delta_{\text{neigh}}$ &  $\Delta_{\text{var}}$ & \hyperref[eq:AGJS]{AGJS} & $\Delta_{\text{neigh}}$ &  $\Delta_{\text{var}}$ \\ \hline
100 & 13.47 & 0.15 & 2.01 & 9.39 & 0.00 & 1.15 \\ \hline
150 & 13.53 & 0.08 & 2.15  & 9.41 & 0.00 & 1.20 \\ \hline
200 & 13.61 & 0.04 & 2.23 & 9.49 & 0.00 & 1.22\\ \hline
\end{tabular}
\caption{Two-type random digraphs with $p_{\text{low}}=0.90$, $q_1=0.70$, $q_2=0.50$. Columns list $q_3\in\{0.05,0.10\}$. Averages based on $100$ graphs for each $(n,q_3)$.}\label{tab:two-type-3}
\end{table}
 
\subsection{A bipartite random digraph model}

We next study a bipartite random digraph model with three parameters:
\begin{itemize}
    \item The number of vertices, $n$.
    \item The proportion of vertices in the first bipartite set, denoted by $a$ where $a\in [0, 1]$.
    \item The probability $p$ that two vertices, one in each of the two partitions, are connected by an arc in a given direction.   
\end{itemize}

\noindent Let $D$ be the resulting bipartite digraph with $V(D)=A\cup B$, where $A$ and $B$ are the two bipartite sets. We set $n_1=|A|=\lceil n\cdot a\rceil$ and $n_2=|B|=n-n_1$. For each ordered pair $(u,v)$ with $u\in A$ and $v\in B$, include the arc $u\to v$ with probability $p$. Independently, for $(v,u)$ with $v\in B$ and $u\in A$, include the arc $v\to u$ with probability $p$. Both directions may occur since 2-cycles are permitted.

We observed the largest improvement when the value of $a$ was small. This is consistent with our earlier observations, because the smaller the value of $a$, the more imbalanced the bipartite graph is. When $a=0.5$, the bipartite graph would have uniform degrees, where we expect no improvement between our bounds and \eqref{eq:AGJS}. Tables \ref{tab:bipartite-1}, \ref{tab:bipartite-2}, and \ref{tab:bipartite-3} show the value of these improvements in a random computer simulation of the bipartite digraph model run with the indicated parameter values. 

\begin{table}[ht]
\centering
\begin{tabular}{|>{\centering\arraybackslash}p{1.5cm}|
                >{\centering\arraybackslash}p{1.5cm}|
                >{\centering\arraybackslash}p{1.5cm}|
                >{\centering\arraybackslash}p{1.5cm}|>{\centering\arraybackslash}p{1.5cm}|>{\centering\arraybackslash}p{1.5cm}|>{\centering\arraybackslash}p{1.5cm}|>{\centering\arraybackslash}p{1.5cm}|}
\hline
$a \rightarrow$ & \multicolumn{3}{|c|}{$0.05$} & \multicolumn{3}{|c|}{$0.10$} \\ \hline
$n \downarrow$ & \hyperref[eq:AGJS]{AGJS} & $\Delta_{\text{neigh}}$ &  $\Delta_{\text{var}}$ & \hyperref[eq:AGJS]{AGJS} & $\Delta_{\text{neigh}}$ &  $\Delta_{\text{var}}$ \\ \hline
100 & 30.82 & 16.94 & 6.09 & 16.12 & 10.01 & 4.89 \\ \hline
150 & 30.59 & 20.62 & 8.04 & 16.61 & 11.28 & 5.66 \\ \hline
200 & 33.55 & 24.30 & 9.56 & 16.87 & 11.97 & 6.09 \\ \hline
\end{tabular}
\caption{Bipartite random digraphs with edge probability $p=0.65$. Columns list $a\in\{0.05,0.10\}$. Averages based on $100$ graphs for each $(n,a)$.}\label{tab:bipartite-1}
\end{table}

\begin{table}[ht]
\centering
\begin{tabular}{|>{\centering\arraybackslash}p{1.5cm}|
                >{\centering\arraybackslash}p{1.5cm}|
                >{\centering\arraybackslash}p{1.5cm}|
                >{\centering\arraybackslash}p{1.5cm}|>{\centering\arraybackslash}p{1.5cm}|>{\centering\arraybackslash}p{1.5cm}|>{\centering\arraybackslash}p{1.5cm}|>{\centering\arraybackslash}p{1.5cm}|}
\hline
$a \rightarrow$ & \multicolumn{3}{|c|}{$0.15$} & \multicolumn{3}{|c|}{$0.20$} \\ \hline
$n \downarrow$ & \hyperref[eq:AGJS]{AGJS} & $\Delta_{\text{neigh}}$ &  $\Delta_{\text{var}}$ & \hyperref[eq:AGJS]{AGJS} & $\Delta_{\text{neigh}}$ &  $\Delta_{\text{var}}$ \\ \hline
100 & 10.65 & 5.99 & 3.77 & 7.88 & 3.62 & 2.95 \\ \hline
150 & 10.65 & 6.31 & 4.07 & 7.99 & 3.88 & 3.18 \\ \hline
200 & 10.98 & 6.81 & 4.36 & 8.06 & 4.02 & 3.29 \\ \hline
\end{tabular}
\caption{Bipartite random digraphs with edge probability $p=0.65$. Columns list $a\in\{0.15,0.20\}$. Averages based on $100$ graphs for each $(n,a)$.}\label{tab:bipartite-2}
\end{table}

\begin{table}[ht]
\centering
\begin{tabular}{|>{\centering\arraybackslash}p{1.5cm}|
                >{\centering\arraybackslash}p{1.5cm}|
                >{\centering\arraybackslash}p{1.5cm}|
                >{\centering\arraybackslash}p{1.5cm}|>{\centering\arraybackslash}p{1.5cm}|>{\centering\arraybackslash}p{1.5cm}|>{\centering\arraybackslash}p{1.5cm}|>{\centering\arraybackslash}p{1.5cm}|}
\hline
$a \rightarrow$ & \multicolumn{3}{|c|}{$0.25$} & \multicolumn{3}{|c|}{$0.30$} \\ \hline
$n \downarrow$ & \hyperref[eq:AGJS]{AGJS} & $\Delta_{\text{neigh}}$ &  $\Delta_{\text{var}}$ & \hyperref[eq:AGJS]{AGJS} & $\Delta_{\text{neigh}}$ &  $\Delta_{\text{var}}$ \\ \hline
100 &  6.25 & 2.06 & 2.32 & 5.21 & 0.97 & 1.81 \\ \hline
150 &  6.23 & 2.12 & 2.42 & 5.26 & 1.05 & 1.90 \\ \hline
200 &  6.35 & 2.28 & 2.53 & 5.28 & 1.10 & 1.95\\ \hline
\end{tabular}
\caption{Bipartite random digraphs with edge probability $p=0.65$. Columns list $a\in\{0.25,0.30\}$. Averages based on $100$ graphs for each $(n,a)$.}\label{tab:bipartite-3}
\end{table}

\subsection{Scale-free random digraph models}\label{subsec:scale-free}

Finally, we consider two scale-free random graph models. These models are designed so that vertices have a probability of attachment proportional to their degree.

\textit{The Barab\'asi--Albert (BA) model.} The BA model is constructed from a fixed parameter $m$. At each iteration, a new vertex is introduced and connected to $m$ existing vertices, with attachment probabilities proportional to vertex degrees. The process terminates when the desired graph order $n$ is reached. Since the BA model is designed for undirected graphs, we adapt it for digraphs by first constructing the undirected graph and then assigning a uniformly random direction to each edge independently.

\textit{The Krapivsky--Rodgers--Redner (KRR) model.} The KRR model~\cite{krapivsky2001}, formulated directly for digraphs, also builds a graph iteratively. At each step~$t$, the model either (a) adds a new vertex with probability~$p$, which attaches to an existing vertex via a directed edge, or (b) adds a directed edge between two existing vertices with probability $1-p$. In both cases, attachment probabilities are proportional to vertex degrees.

The experimental results for both scale-free models are reported in Tables~\ref{tab:BA-1}--\ref{tab:BA-2} and Tables~\ref{tab:KRR-1}--\ref{tab:KRR-2}. Neither the neighborhood-based bound nor the variance-based bound shows a meaningful improvement over the AGJS bound~\eqref{eq:AGJS}.

\medskip 

\begin{table}[ht]
\centering
\begin{tabular}{|>{\centering\arraybackslash}p{1.5cm}|
                >{\centering\arraybackslash}p{1.5cm}|
                >{\centering\arraybackslash}p{1.5cm}|
                >{\centering\arraybackslash}p{1.5cm}|>{\centering\arraybackslash}p{1.5cm}|>{\centering\arraybackslash}p{1.5cm}|>{\centering\arraybackslash}p{1.5cm}|>{\centering\arraybackslash}p{1.5cm}|}
\hline
$m \rightarrow$ & \multicolumn{3}{|c|}{$1$} & \multicolumn{3}{c|}{$5$} \\ \hline
$n \downarrow$ & \hyperref[eq:AGJS]{AGJS} & $\Delta_{\text{neigh}}$ &  $\Delta_{\text{var}}$ & \hyperref[eq:AGJS]{AGJS} & $\Delta_{\text{neigh}}$ &  $\Delta_{\text{var}}$ \\ \hline
100 & 88.98 & 0.00 & 0.05 & 35.22 & 0.00 & 0.50 \\ \hline
200 & 178.16 & 0.00 & 0.05  & 70.57 & 0.02 & 0.55 \\ \hline
300 & 267.59 & 0.00 & 0.05 & 105.76 & 0.05 & 0.60 \\ \hline
\end{tabular}
\caption{Barab\'asi--Albert digraphs. Averages based on $100$ graphs for each $(n,m)$.}\label{tab:BA-1}
\end{table}

\begin{table}[ht]
\centering
\begin{tabular}{|>{\centering\arraybackslash}p{1.5cm}|
                >{\centering\arraybackslash}p{1.5cm}|
                >{\centering\arraybackslash}p{1.5cm}|
                >{\centering\arraybackslash}p{1.5cm}|>{\centering\arraybackslash}p{1.5cm}|>{\centering\arraybackslash}p{1.5cm}|>{\centering\arraybackslash}p{1.5cm}|>{\centering\arraybackslash}p{1.5cm}|}
\hline
$m \rightarrow$ & \multicolumn{3}{|c|}{$10$} & \multicolumn{3}{c|}{$20$} \\ \hline
$n \downarrow$ & \hyperref[eq:AGJS]{AGJS} & $\Delta_{\text{neigh}}$ &  $\Delta_{\text{var}}$ & \hyperref[eq:AGJS]{AGJS} & $\Delta_{\text{neigh}}$ &  $\Delta_{\text{var}}$ \\ \hline
100 & 18.81 & 0.00 & 0.57 & 10.32 & 0.01 & 0.58 \\ \hline
200 & 37.15 & 0.01 & 0.65  & 19.17 & 0.01 & 0.64 \\ \hline
300 & 55.84 & 0.01 & 0.70 & 28.55 & 0.00 & 0.69 \\ \hline
\end{tabular}
\caption{Barab\'asi--Albert digraphs. Averages based on $100$ graphs for each $(n,m)$.}\label{tab:BA-2}
\end{table}

\begin{table}[ht]
\centering
\begin{tabular}{|>{\centering\arraybackslash}p{1.5cm}|
                >{\centering\arraybackslash}p{1.5cm}|
                >{\centering\arraybackslash}p{1.5cm}|
                >{\centering\arraybackslash}p{1.5cm}|>{\centering\arraybackslash}p{1.5cm}|>{\centering\arraybackslash}p{1.5cm}|>{\centering\arraybackslash}p{1.5cm}|>{\centering\arraybackslash}p{1.5cm}|}
\hline
$p \rightarrow$ & \multicolumn{3}{|c|}{$0.2$} & \multicolumn{3}{c|}{$0.4$} \\ \hline
$t \downarrow$ & \hyperref[eq:AGJS]{AGJS} & $\Delta_{\text{neigh}}$ &  $\Delta_{\text{var}}$ & \hyperref[eq:AGJS]{AGJS} & $\Delta_{\text{neigh}}$ &  $\Delta_{\text{var}}$ \\ \hline
100 & 11.53 & 0.04 & 0.18 & 30.71 & 0.03 & 0.12 \\ \hline
200 & 24.05 & 0.13 & 0.17  & 59.75 & 0.05 & 0.13 \\ \hline
300 & 37.56 & 0.20 & 0.17 & 91.35 & 0.05 & 0.13 \\ \hline
\end{tabular}
\caption{KRR digraphs with $t$ iterations. Averages based on $100$ graphs for each $(t,p)$.}\label{tab:KRR-1}
\end{table}

\begin{table}[ht]
\centering
\begin{tabular}{|>{\centering\arraybackslash}p{1.5cm}|
                >{\centering\arraybackslash}p{1.5cm}|
                >{\centering\arraybackslash}p{1.5cm}|
                >{\centering\arraybackslash}p{1.5cm}|>{\centering\arraybackslash}p{1.5cm}|>{\centering\arraybackslash}p{1.5cm}|>{\centering\arraybackslash}p{1.5cm}|>{\centering\arraybackslash}p{1.5cm}|}
\hline
$p \rightarrow$ & \multicolumn{3}{|c|}{$0.6$} & \multicolumn{3}{c|}{$0.8$} \\ \hline
$t \downarrow$ & \hyperref[eq:AGJS]{AGJS} & $\Delta_{\text{neigh}}$ &  $\Delta_{\text{var}}$ & \hyperref[eq:AGJS]{AGJS} & $\Delta_{\text{neigh}}$ &  $\Delta_{\text{var}}$ \\ \hline
100 & 49.71 & 0.02 & 0.10 & 69.22 & 0.01 & 0.10 \\ \hline
200 & 96.87 & 0.02 & 0.12  & 136.23 & 0.00 & 0.11 \\ \hline
300 & 144.94 & 0.01 & 0.13 & 203.34& 0.00 & 0.12 \\ \hline
\end{tabular}
\caption{KRR digraphs with $t$ iterations. Averages based on $100$ graphs for each $(t,p)$.}\label{tab:KRR-2}
\end{table}

\bibliographystyle{amsplain}
\bibliography{references}

\end{document}